\documentclass[12pt]{article}
\usepackage[utf8]{inputenc}
\usepackage{cite}
\usepackage{amsthm}
\usepackage{amsmath}
\usepackage{amssymb}
\usepackage{yfonts}
\usepackage[colorlinks]{hyperref}
\usepackage{authblk}
\usepackage{graphicx}
\usepackage{multicol}
\usepackage{comment}
\usepackage{blkarray}
\newtheorem{thm}{Theorem}[section]
\newtheorem{lem}[thm]{Lemma}
\newtheorem{prop}[thm]{Proposition}
\newtheorem{cor}[thm]{Corollary}

\theoremstyle{definition}
\newtheorem{claim}{Claim}[section]

\newtheorem{conj}{Conjecture}[section]

\theoremstyle{remark}
\newtheorem{rem}{Remark}

\begin{document}

\title{Algebraic sets defined by the commutator matrix}
\author[1]{Zhibek Kadyrsizova}
\author[2]{Madi Yerlanov}
\affil[1]{{\small Department of  Mathematics, Nazarbayev University, \newline
email: zhibek.kadyrsizova@nu.edu.kz}}
\affil[2]{{\small Department of Mathematics, Simon Fraser University, \newline
email: myerlano@sfu.ca}}

\date{}
\maketitle
\begin{abstract}
In this paper we study algebraic sets of pairs of matrices defined by the vanishing of either the diagonal of their commutator matrix or its anti-diagonal. We find a system of parameters for the coordinate rings of these two sets and their intersection and show that they are complete intersections. Moreover, we prove that these algebraic sets are $F$-pure over a field of positive prime characteristic and the algebraic set of pairs of matrices with the zero diagonal commutator is $F$-regular. \\

\end{abstract}

\section{Introduction}

Let $X=\left(x_{ij}\right)_{1\leq i, \, j \leq n}$  and $Y=\left(y_{ij}\right)_{1\leq i, \, j \leq n}$ be square matrices of size $n\geq 3$ with indeterminate entries over a field $K$ and $R=K[X, Y] $ be the polynomial ring over $K$ in $\{x_{ij},\, y_{ij}\, |\, 1\leq i, \, j \leq n\}$. 
Let $\mathfrak{m}$ be the homogeneous maximal ideal of $R$. 
Let $C=XY-YX=\left(c_{ij}\right)$ be the commutator matrix and $\mathcal{I}$ be the ideal of $R$ generated by the diagonal and anti-diagonal entries of $C$.  
Let $I$ be the ideal of $R$ generated by the diagonal entries of $C$ and $J$ be the ideal of $R$ generated by the anti-diagonal entries of $C$, that is, $\mathcal{I}=I+J$.  
Let $X_0$ and $Y_0$ be the square matrices of size $(n-2)$ obtained by removing the first and last rows and columns of $X$ and $Y$, respectively. 
We denote by $\mathcal{I}_0$ the ideal of $R_0=K[X_0,Y_0]$ generated by the diagonal and anti-diagonal entries of $C_0=X_0Y_0-Y_0X_0=\left(c_{ij}^{0}\right)$. 

\begin{rem}
Observe that if $n=2$, then $\mathcal{I}$ is the ideal generated by all the entries of the commutator matrix $C$. Therefore, the vanishing set defines the variety of commuting matrices for $n=2$ and also a determinantal ideal,\cite{det}, \cite{bruns}, \cite{diag}. Since the properties of rings that we study in this paper have been extensively studied in this case, we focus on matrices of size $n\geq 3$.
\end{rem}

\begin{rem}
Since the trace of the commutator matrix $C$ is equal to 0, from the main diagonal of $C$ it is sufficient to use any $n-1$ entries. Therefore, the ideal $\mathcal{I}$ can be generated by either $2n-2$, if $n$ is odd, and $2n-1$, if $n$ is even, elements. In fact, as the results of the paper show, this defines the minimal number of generators for the ideal. 
\end{rem}

Below is a summary of the results that we prove in the paper.
\begin{thm} 
\begin{itemize}
    \item[1.] The rings $R/\mathcal{I}$, $R/I$ and $R/J$ are complete intersections, that is, the ideals $\mathcal{I}$, $I$ and $J$ are generated by a regular sequence.  
    \item [2.] $R/\mathcal{I}$ and $R/J$ are $F$-pure rings when the characteristic of the field $K$ is positive prime. In particular, in this case $\mathcal{I}$ and $J$ are radical ideals.  
    \item[3.] $R/I$ is $F$-regular when the characteristic of the field $K$ is positive prime. In particular, in this case $R/I$ is an integral domain. 
\end{itemize}
\end{thm}

\begin{rem}
Notice that the fact $R/I$ is a complete intersection is first proved by H.W.Young in \cite{young}, here we obtain this fact as a corollary of Theorems \ref{SOPodd}, \ref{SOPeven} and \ref{SOPchar2}. Moreover, Young proved that $R/I$ is reduced when $K$ is a field of characteristic 0.  
\end{rem}

\section{A system of parameters}

In this section we find a homogeneous system of parameters for $R/\mathcal{I}$. We do the proof by induction on $n$ and with the induction step equal to 2. Therefore, we prove the results separately for odd  and even values of $n$. Moreover, when $n$ is even we consider two cases: when the characteristic of $K$ is not equal to 2 and when the characteristic of $K$ is 2. There is an appendix at the end of the paper which the reader may find helpful in understanding the pattern of the system of parameters as $n$ varies.  


\begin{thm} \label{SOPodd}
Let $n=2k+1$ be an odd integer with $k \geq 1$.
Let \[\Lambda=\left\{(i,j) \, | 1 \leq i\leq k, \, 1 \leq j \leq 2k+1 \right\}\cup \]
\[\left\{(k+1,k+1) \right\}\cup\]
\[ \left\{(i,j) \, | k+2 \leq i\leq 2k+1, \, k+2 \leq j \leq 2k+1 \right\}\cup\]
\[ \left\{(i,j) \, | k+3 \leq i\leq 2k+1, \, 1 \leq j \leq k \right\}\] 

Then \[x_{ij}, \, y_{ji} \text{ with } (i,j)\in \Lambda,\]
\[x_{k+1, s}-y_{2k+2-s, k+1} \text{ with } 1 \leq s \leq 2k+1\]
(notice $s=k+1$ is included above in $\Lambda$)
\[x_{s, k+1}-y_{k+1, s} \text{ with } k+2 \leq s \leq 2k+1\]
\[x_{k+2, s}-y_{s, k+2} \text{ with } 1 \leq s \leq k\]
is a homogeneous system of parameters for $R/\mathcal{I}$. 
\end{thm}
\begin{proof}
We prove the theorem by showing that the generators of $\mathcal{I}$, $n-1$ entries of the  main diagonal and $n-1$ entries of the anti-diagonal of $C$, together with the ring elements from the statement of the theorem form a system of parameters for $R$.  We achieve this by showing that the ideal they generate is $\mathfrak{m}$-primary. It is not difficult to compute that their number is exactly equal to $2n^2$, the dimension of $R$. Moreover, we use induction on $k$.\\

Let $k=1$. We prove the base of the induction by showing that the radical of the ideal generated by $\mathcal{I}$ and the above sequence has Krull dimension $18$, the dimension of $R$ in this case. Observe, that $\mathcal{I}$ is generated by 4 elements and there are 14 elements in the sequence. This will show that the generators of $\mathcal{I}$ and the sequence elements form a system of parameters for $R$. 

Factor out $R$ by the ideal generated by the sequence elements. Then we have the following matrices
\[\overline{X}=\left[\begin{array}{ccc}
0& 0&0 \\
x_{21}& 0&x_{23} \\
x_{31} & x_{32}&0
\end{array}\right] \text{ and } \overline{Y}=\left[\begin{array}{ccc}
0& x_{23}&x_{31} \\
0& 0&x_{32} \\
0& x_{21}&0
\end{array}\right]\]
and the commutator matrix becomes 
\[\overline{XY-YX}=\left[\begin{array}{ccc}
-x_{21}x_{23}-x_{31}^2& \star&-x_{23}^2 \\
\star& 2x_{21}x_{23}-x_{32}^2&\star \\
-x_{21}^2& \star&x_{31}^2+x_{32}^2-x_{21}x_{23}
\end{array}\right]\]
 Therefore, 
 \[\text{Rad }\overline{\mathcal{I}}=\text{Rad}(-x_{23}^2, -x_{21}^2, -x_{21}x_{23}-x_{31}^2, 2x_{21}x_{23}-x_{32}^2, x_{31}^2+x_{32}^2-x_{21}x_{23})=\]
 \[(x_{21}, x_{23}, x_{31}, x_{32})\]
 is the homogeneous maximal ideal in $K[x_{21}, x_{23}, x_{31}, x_{32}]$. Thus we have the result for $n=3$.

Now we prove the induction step. We claim that a similar pattern persists for all positive integer values of $k$. 
We apply the induction hypothesis in the ring $R_0=K[X_0, Y_0]$.  
Factor out $R$ by the ring elements in the set defined in the statement of the theorem. Then we have the following matrices 
\begin{small}\[\overline{X}=
\begin{blockarray}{cccccc}
\begin{block}{[c|ccc|c]c}
  0& 0&\ldots &0&0 \\ 
  \cline{1-6}
  0 &&&&0\\
  \vdots &&&& \vdots \\
     0 &&&& 0 \\
x_{k+1,1} &&& & x_{k+1, 2k+1} \\
  x_{k+2,1} &&\raisebox{-5pt}{{\Large \mbox{{$X_0$}}}}&& 0 \\
   0 &&&& 0 \\
  \vdots &&&& \vdots \\
   0 &&&& 0 \\
  \cline{1-6}
   0&0 \ldots 0&x_{2k+1, k+1}&0 \ldots 0&0\\ 
 \end{block}
\end{blockarray}\]
and 
\[\overline{Y}=
\begin{blockarray}{cccccc}
&&(k+1)\text{-column }\hspace{3mm} &&\\
&&\swarrow \quad \hspace{3mm}&&\\
\begin{block}{[c|ccc|c]c}
 0&0 \ldots 0&x_{k+1, 2k+1}\quad x_{k+2, 1}&0 \ldots 0&0\\  
  \cline{1-5}
    0 &&&&0\\
\vdots &&&& \vdots \\
     0 &&&& 0 \\
0 & &\raisebox{-5pt}{{\Large \mbox{{$Y_0$}}}}&& x_{2k+1, k+1}&\leftarrow (k+1)\text{-row} \\
   0 &&&& 0 \\
  \vdots &&&& \vdots \\
   0 &&&& 0 \\
  \cline{1-5}
   0&0 \ldots 0&x_{k+1, 1} \quad \quad 0&0 \ldots 0&0\\  
   \end{block}
\end{blockarray}
 \]
\end{small}

We start off by writing out the generators of $\mathcal{I}$ in the factor ring. 

\[\overline{c}_{11}=-x_{k+1, 1}x_{k+1, 2k+1}-x_{k+2, 1}^2\]
\[\overline{c}_{1, 2k+1}=-x_{k+1, 2k+1}^2\]
\[\overline{c}_{ 2k+1, 1}=-x_{k+1,1}^2\]
\[\overline{c}_{2k+1, 2k+1}=x_{2k+1, k+1}^2-x_{k+1,1}x_{k+1, 2k+1}\]
For $2 \leq i\leq k$,
\[\overline{c}_{ii}=\sum_{j=1, j\neq i}^{2k+1}(x_{ij}y_{ji}-x_{ji}y_{ij})=-x_{k+1, i}y_{i, k+1}-x_{k+2, i}y_{i, k+2}=\]
\[-x_{k+1, i}x_{k+1, 2k+2-i}-x_{k+2, i}^2=\overline{c}_{i-1,i-1}^0\]
\[\overline{c}_{k+1, k+1}=\sum_{j=2, j\neq k+1}^{2k}(x_{j+1,j}y_{j, k+1}-x_{j,k+1}y_{k+1,j})+\]
\[x_{k+1, 1}y_{1, k+1}-x_{k+1, 2k+1}y_{2k+1,k+1}-x_{2k+1, k+1}y_{k+1, 2k+1}=\]
\[=\overline{c}_{k, k}^0+2x_{k+1, 1}x_{k+1, 2k+1}-x_{2k+1, k+1}^2\]
\[\overline{c}_{k+2, k+2}=\sum_{j=2}^{2k}(x_{ij}y_{ji}-x_{ji}y_{ij})+x_{k+2,1}^2=\overline{c}_{k+1, k+1}^0+x_{k+2, 1}^2\]
For $k+3\leq i\leq 2k$
\[\overline{c}_{ii}=\sum_{j=2, j\neq i}^{2k}(x_{ij}y_{ji}-x_{ji}y_{ij})=\overline{c}_{i-1,i-1}^0\]

For $2\leq i\leq 2k$ and $i\neq k+1$
\[\overline{c}_{i, 2k+2-i}=\overline{c}_{i-1, 2k+1-i}^0\]

Now we have that \[\overline{\mathcal{I}}=\left(\overline{c}_{ii}, \overline{c}_{i, 2k+2-i}\right)_{1\leq i\leq 2k+1}=\]
\[(-x_{k+1, 1}x_{k+1, 2k+1}-x_{k+2, 1}^2, -x_{k+1, 2k+1}^2, -x_{k+1,1}^2,\]
\[  \overline{c}_{k, k}^0+2x_{k+1, 1}x_{k+1, 2k+1}-x_{2k+1, k+1}^2,\]
\[\overline{c}_{k+1, k+1}^0+x_{k+2, 1}^2, x_{2k+1, k+1}^2-x_{k+1, 1}x_{k+1, 2k+1})+\]
\[(\overline{c}_{ii}^0)_{1\leq i\leq 2k-1,\, i\neq k, k+1}+(\overline{c}_{i-1, 2k+1-i}^0)_{2\leq i\leq 2k, \, i\neq k+1}.\]

Then \[\text{Rad}(\overline{\mathcal{I}})=(x_{k+1, 2k+1}, x_{k+2, 1}, x_{k+1, 1},x_{2k+1, k+1})+\text{Rad}(\overline{c}_{ii}^0, \overline{c}_{i, 2k-i}^0)_{1\leq i\leq 2k-1}.\]
By induction hypothesis, $\text{Rad}(\overline{c}_{ii}^0, \overline{c}_{i, 2k-i}^0)_{1\leq i\leq 2k-1}$ is the homogeneous maximal ideal in $\overline{R}_0=K[\overline{X}_0, \overline{Y}_0]$.
Therefore, we have that $\text{Rad}(\overline{\mathcal{I}})$ is the homogeneous maximal ideal in $\overline{R}=K[\overline{X}, \overline{Y}]$.
   \end{proof}
   
\begin{thm} \label{SOPeven}
Let $n=2k$ be an even integer with $k\geq 2$. Let $K$ be a field of characteristic not equal to 2. 
Let \[\Omega=\left\{(i,j) \, | 1 \leq i\leq k-1, \, 1 \leq j \leq 2k \right\}\cup \]
\[\left\{(k,k) \right\}\cup \left\{(k,k+1) \right\}\cup \left\{(k+1,k+1) \right\}\cup\]
\[ \left\{(k+1,j) \, |  \, k+3 \leq j \leq 2k \right\}\cup\]
\[ \left\{(i,k+1) \, |  \, k+3 \leq i \leq 2k \right\}\cup\]
\[ \left\{(i,j) \, | k+2 \leq i\leq 2k, \, 1 \leq j \leq k-1 \right\}\cup\]
\[ \left\{(i,j) \, | k+2 \leq i\leq 2k, \, k+2 \leq j \leq 2k \right\}\]
 
Then \[x_{ij}, \, y_{ji} \text{ with } (i,j)\in \Omega,\]
\[x_{k, s}-y_{2k+1-s, k}\]\[ \text{ with } 1 \leq s \leq 2k \text{ (notice the cases $s=k, k+1$ are included above in }\Omega ),\]
\[x_{k+1, s}-y_{s, k+1}\]\[ \text{ with } 1 \leq s \leq k+2\text{ (notice the case $s=k+1$ is included above in }\Omega ),\]
\[x_{k+2, k}-y_{k+1, k+2}, \quad x_{k+2, k+1}-y_{k, k+2},\]
\[x_{s, k}-y_{k, s} \text{ with } k+3 \leq s \leq 2k\]
is a  homogeneous system of parameters for $R/\mathcal{I}$. 
\end{thm}
\begin{proof}
As in the previous theorem for matrices of odd size, we prove the theorem by showing that the generators of $\mathcal{I}$, $n-1$ entries of the  main diagonal and $n$ entries of the anti-diagonal of $C$, together with the ring elements from the statement of the theorem form a system of parameters for $R$. We achieve this by showing that the ideal they generate is $\mathfrak{m}$-primary. It is not difficult to compute that their number is exactly equal to $2n^2$, the dimension of $R$. Moreover, we use induction on $k$.\\

Let $k=2$. The idea of the proof is similar to that of the proof of the case $n=3$. Factor out by the ideal generated by the elements of the sequence. Then we have the following matrices
\[\overline{X}=\left[\begin{array}{cccc}
0& 0&0&0\\
x_{21}& 0&0&x_{24}\\
x_{31}&x_{32}&0&x_{34}\\
0&x_{42}&x_{43}&0\end{array}\right] \text{ and } \overline{Y}=\left[\begin{array}{cccc}
0& x_{24}&x_{31}&0\\
0& 0&x_{32}&x_{43}\\
0&0&0&x_{42}\\
0&x_{21}&x_{34}&0\\
\end{array}\right]\]
and the commutator matrix becomes 

$$\overline{XY-YX}=$$\\
\tabcolsep=0.001cm
\scalebox{0.8}{$\left[\begin{array}{cccc}
-x_{21}x_{24}-x_{31}^2&\star&\star&-x_{31}x_{34}-x_{24}^2\\
\star&2x_{21}x_{24}-x_{32}^2-x_{42}x_{43}&x_{21}x_{31}-x_{24}x_{34}-x_{43}^2&\star\\
\star&x_{24}x_{31}-x_{21}x_{34}-x_{42}^2&x_{31}^2+x_{32}^2+x_{34}^2-x_{42}x_{43}&\star\\
-x_{31}x_{34}-x_{21}^2&\star&\star&2x_{42}x_{43}-x_{21}x_{24}-x_{34}^2\\
\end{array}\right]$}\\

We show that the ideal generated by the diagonal and anti-diagonal entries  of the matrix above is primary to the maximal ideal in the corresponding polynomial ring $K[x_{21}, x_{24}, x_{31}, x_{32}, x_{34}, x_{42}, x_{43}]$. We achieve this by using homogeneous Nullstellensatz theorem. That is, we show that the zero set of the ideal is trivial. 

First, observe that $x_{21}^2=x_{24}^2$ and $x_{31}^4=x_{21}^2x_{24}^2$. Hence, $x_{31}^4=x_{31}^2x_{34}^2$ and either $x_{31}=0$ or $x_{31}^2=x_{34}^2$. In the first case, it is not hard to see that this implies that we get the trivial solution. 
In the second case we have that $0=2x_{32}^2+x_{31}^2+x_{34}^2-2x_{21}x_{24}=2x_{32}^2-4x_{31}^2$. Since the characteristic of the field is not 2, we obtain that $x_{32}^2=2x_{31}^2$.  Then $x_{42}x_{43}=4x_{31}^2$ and we get that $0=2x_{42}x_{43}-x_{21}x_{24}-x_{34}^2=8x_{31}^2$, which implies that $x_{31}=0$ and again we have the trivial solution only. 

Now we are ready to prove the general statement of the theorem. From now assume that $k\geq 3$ and for all pairs of matrices of size smaller than $2k$, the statement of the theorem holds. Factor $R$ by the ideal generated by the ring elements from the statement. Then we have the following matrices
\begin{small}\[\overline{X}=
\begin{blockarray}{cccccc}
\begin{block}{[c|ccc|c]c}
  0& 0&\ldots &0&0 \\ 
  \cline{1-6}
  0 &&&&0\\
  \vdots &&&& \vdots \\
     0 &&&& 0 \\
x_{k,1} & &&& x_{k, 2k} \\
  x_{k+1,1} &&\raisebox{-5pt}{{\Large \mbox{{$X_0$}}}}&& 0 \\
  0 &&&& 0 \\
  \vdots &&&& \vdots \\
   0 &&&& 0 \\
  \cline{1-6}
   0&0 \ldots 0&x_{2k, k}&0 \ldots 0&0\\ 
    \end{block}
\end{blockarray}\]
and 
\[\overline{Y}=
\begin{blockarray}{cccccc}
&&k\text{-column }\hspace{3mm} &&\\
&&\swarrow \quad \hspace{3mm}&&\\
\begin{block}{[c|ccc|c]c}
 0&0 \ldots 0&x_{k, 2k}\quad x_{k+1, 1}&0 \ldots 0&0\\  
 \cline{1-5}
    0 &&&&0\\
\vdots &&&& \vdots \\
     0 &&&& 0 \\
0 && \raisebox{-5pt}{{\Large \mbox{{$Y_0$}}}}&& x_{2k, k} &\leftarrow k\text{-row} \\
   0 &&&& 0 \\
  \vdots &&&& \vdots \\
   0 &&&& 0 \\
  \cline{1-5}
   0&0 \ldots 0&x_{k,1}\quad \quad 0&0 \ldots 0&0\\  \end{block}
\end{blockarray} \]
\end{small}

\[\overline{c}_{11}=\sum_{j=2}^{2k}(x_{1j}y_{j1}-x_{j1}y_{1j})=-x_{k, 1}x_{k, 2k}-x_{k+1, 1}^2\]
\[\overline{c}_{1,2k}=-x_{k, 2k}^2\]
\[\overline{c}_{2k,1}=-x_{k, 1}^2\]
\[\overline{c}_{2k,2k}=x_{2k, k}^2-x_{k, 1}x_{k, 2k}\]
For $2\leq i\leq k-1$
\[\overline{c}_{ii}=\overline{c}_{i-1,i-1}^0\]
\[\overline{c}_{kk}=2x_{k, 1}x_{k, 2k}-x_{2k, k}^2+\overline{c}_{k-1,k-1}^0\]
\[\overline{c}_{k+1,k+1}=x_{k+1, 1}^2+\overline{c}_{k, k}^0\]
For $k+2 \leq i \leq 2k-1$, \[\overline{c}_{ii}=\overline{c}_{i-1,i-1}^0, \] \\

For $2\leq i\leq k-1$ and $k+2\leq i \leq 2k-1$, \[\overline{c}_{i, 2k+1-i}=\overline{c}_{i-1, 2k-i}^0,\] 
\[\overline{c}_{k,k+1}=x_{k, 1}x_{k+1, 1}+\overline{c}_{k-1, k}^0, \]
\[\overline{c}_{k+1,k}=x_{k+1, 1}x_{k, 2k}+\overline{c}_{k, k-1}^0.\]
Now we have that \[\overline{\mathcal{I}}=(-x_{k, 1}x_{k, 2k}-x_{k+1, 1}^2,-x_{k, 2k}^2, -x_{k, 1}^2, x_{2k, k}^2-x_{k, 1}x_{k,2k}, x_{k, 1}x_{k+1, 1}+\overline{c}_{k-1, k}^0, \]
\[ x_{k+1, 1}x_{k, 2k}+\overline{c}_{k, k-1}^0)+
(\overline{c}_{ii}^0)_{1\leq i\leq k-2,\, k+1\leq i \leq 2k-2}\, +(\overline{c}_{i, 2k+1-i}^0)_{1\leq i\leq k-2,\, k+1\leq i\leq 2k-2}\]

Then $\text{Rad}(\overline{\mathcal{I}})=(x_{k, 2k}, x_{k, 1}, x_{2k, k},x_{k+1, 1})+\text{Rad}\left(\overline{c}_{ii}^0, \overline{c}_{i, 2k+1-i}^0\right)_{1\leq i\leq 2k-2}$. 
By induction hypothesis, $\text{Rad}(\overline{c}_{ii}^0,\overline{c}_{i, 2k+1-i}^0)_{1\leq i\leq 2k-2}$ is the homogeneous maximal ideal in $\overline{R}_0=K[\overline{X}_0, \overline{Y}_0]$. 
Therefore, we have that $\text{Rad}(\overline{\mathcal{I}})$ is the homogeneous maximal ideal in $\overline{R}=K[\overline{X}, \overline{Y}]$. 

\end{proof}

\begin{thm}\label{SOPchar2}
Let $n=2k$ be an even integer with $k\geq 2$. Let $K$ be a field of characteristic 2. 
Let \[\Omega_2=\left\{(i,j) \, | \, 1 \leq i\leq k-1, \, 1 \leq j \leq 2k \right\}\cup \]
\[\left\{(k,j) \, |\,   k-1 \leq j \leq k+1 \right\}\cup \]
\[\left\{(k+1,k+1), \, (k+2,k+2) \right\}\cup \]
\[\left\{(k+1,j) \, | \, k+3 \leq j \leq 2k \right\}\cup \]
\[\left\{(k+2,j) \, | \, 1 \leq j \leq k-2 \right\}\cup \]
\[\left\{(k+2,j) \, | \, k+2 \leq j \leq 2k \right\}\cup \]
\[\left\{(i,j) \, |\, k+3\leq i\leq 2k, \,  1 \leq j \leq k-1 \right\}\cup \]
\[\left\{(i,j) \, |\, k+3\leq i\leq 2k, \,  k+1 \leq j \leq 2k \right\} \]
Then \[x_{ij}, \, y_{ji} \text{ with } (i,j)\in \Omega_2,\]
\[x_{k, s}-y_{2k+1-s, k} \text{ with } 1 \leq s \leq k-2 \text{ and } k+3 \leq s \leq 2k,\]
\[x_{k+1, s}-y_{s, k+1} \text{ with } 1 \leq s \leq k-2,\]
\[x_{s, k}-y_{k, s} \text{ with } k+3 \leq s \leq 2k,\]
\[x_{k, k+2}-y_{k+2, k},\,  x_{k+1, k-1}-y_{k+2, k+1}, \, x_{k+1, k}-y_{k, k+1}\]
\[x_{k+1, k+2}-y_{k-1, k+1}, \, x_{k+2, k-1}-y_{k-1, k+2}, x_{k+2, k}-y_{k+1, k+2}, x_{k+2, k+1}-y_{k, k+2}\]
is a  homogeneous system of parameters for $R/\mathcal{I}$. 
\end{thm}
\begin{proof}
We use the same approach as in the previous two theorems. 

First, we show that the theorem is true for $k=2$. 

After factoring out $R$ by the elements of the sequence we obtain the following matrices
\[\overline{X}=\left[\begin{array}{cccc}
0& 0&0&0\\
0& 0&0&x_{24}\\
x_{31}&x_{32}&0&x_{34}\\
x_{41}&x_{42}&x_{43}&0\end{array}\right] \text{ and } \overline{Y}=\left[\begin{array}{cccc}
0& 0&x_{34}&x_{41}\\
0& 0&x_{32}&x_{43}\\
0&0&0&x_{42}\\
0&x_{24}&x_{31}&0\\
\end{array}\right]\]
and the commutator matrix becomes 

$$\overline{XY-YX}=$$\\
\tabcolsep=0.001cm
\scalebox{1}{$\left[\begin{array}{cccc}
x_{31}x_{34}+x_{41}^2&\star&\star&x_{34}^2\\
\star&x_{24}^2+x_{32}^2+x_{42}x_{43}&x_{24}x_{31}+x_{43}^2&\star\\
\star&x_{24}x_{34}+x_{42}^2&x_{32}^2+x_{42}x_{43}&\star\\
x_{31}^2&\star&\star&x_{41}^2+x_{24}^2+x_{31}x_{34}\\
\end{array}\right]$}\\

Similarly to the above proofs, it is not difficult to show that $\text{Rad}(\overline{\mathcal{I}})$ is equal to the homogeneous maximal ideal in the quotient ring of $R$ modulo the ideal generated by the sequence elements. 

The induction step follows exactly the same lines as the induction step in Theorem \ref{SOPeven}. 
\end{proof}

\begin{thm}
The ideal $\mathcal{I}$ is generated by a regular sequence, that is, $R/\mathcal{I}$ is a complete intersection of dimension $2n^2-2n+2$, when $n$ is odd, and $2n^2-2n+1$, when $n$ is even. 
\end{thm}
\begin{proof}
The ideal $\mathcal{I}$ is generated by $2n-2$ and $2n-1$ elements when $n$ is odd and even, respectively. In the above three theorems we found a system of parameters on $R/\mathcal{I}$ of length $2n^2-2n+2$ and $2n^2-2n+1$,  respectively. Therefore, the height of $\mathcal{I}$ is equal to the number of its generators. Thus the result. 
\end{proof}
\textbf{Assumption.} For the rest of the paper, when we say "part of a system of parameters for $R/\mathcal{I}$" we mean part of a system of parameters from one of the Theorems \ref{SOPodd}, \ref{SOPeven}, \ref{SOPchar2} such that those ring elements from the system which identify/associate $x$'s and $y$'s are omitted. 

 \section{F-purity}

In this section we prove that the algebraic set of pairs of matrices defined by the vanishing of the main diagonal and the main anti-diagonal of their commutator matrix is $F$-pure for matrices of all sizes and in all positive prime characteristics. We do so by using the fact that the algebraic set is a complete intersection and that $F$-purity deforms for Gorenstein rings, \cite{fedder}, Theorem 3.4. That is, we prove the result by showing that the ring $R/\mathcal{I}$ is $F$-pure once we factor it by a regular sequence.  

Similarly to the above section, we split our proof into two cases based on the parity of $n$ and proceed by induction on the size of the matrices. Therefore, we first state several results for small values of $n$. 

\begin{prop} Let $n=3$ and $K$ be a field of positive prime characteristic $p$. Then $R/\mathcal{I}$ is an $F$-pure ring.  \end{prop}
\begin{proof}
By factoring out part of the system of parameters from Theorem \ref{SOPodd} we obtain the following matrices.  
\[\overline{X}=\left[\begin{array}{ccc}
0& 0&0 \\
x_{21}& 0&x_{23} \\
x_{31} & x_{32}&0
\end{array}\right] \text{ and } \overline{Y}=\left[\begin{array}{ccc}
0& y_{12}&y_{13} \\
0& 0&y_{23} \\
0& y_{32}&0
\end{array}\right]\]
and the commutator matrix becomes 
\[\overline{XY-YX}=\]\[\left[\begin{array}{ccc}
-x_{21}y_{12}-x_{31}y_{13}& \star&-x_{23}y_{12} \\
\star& x_{21}y_{12}+x_{23}y_{32}-x_{32}y_{23}&\star \\
-x_{21}y_{32}& \star&x_{31}y_{13}+x_{32}y_{23}-x_{23}y_{32}
\end{array}\right].\]
Therefore, the image of the ideal $\overline{\mathcal{I}}=(x_{21}y_{32}, x_{23}y_{12}, x_{21}y_{12}+x_{31}y_{13}, x_{31}y_{13}+x_{32}y_{23}-x_{23}y_{32})$. 
Then since $\mathcal{I}$ is a complete intersection ideal, we have that \[\overline{\mathcal{I}}^{[p]}:\overline{\mathcal{I}}=(x_{21}y_{32}x_{23}y_{12}(x_{21}y_{12}+x_{31}y_{13})( x_{31}y_{13}+x_{32}y_{23}-x_{23}y_{32}))^{p-1}+\overline{\mathcal{I}}^{[p]}=\]
\[\left(x_{21}^{p-1}x_{23}^{p-1}y_{12}^{p-1}y_{32}^{p-1}(x_{21}y_{12}+x_{31}y_{13})^{p-1}( x_{31}y_{13}+x_{32}y_{23}-x_{23}y_{32}))^{p-1}\right)+\overline{\mathcal{I}}^{[p]}\]
\[\not\subset \mathfrak{m}^{[p]}\]
since 
\[x_{21}^{p-1}x_{23}^{p-1}y_{12}^{p-1}y_{32}^{p-1}(x_{21}y_{12}+x_{31}y_{13})^{p-1}( x_{31}y_{13}+x_{32}y_{23}-x_{23}y_{32}))^{p-1}\]
has a monomial term 
\[(x_{21}x_{23}y_{12}y_{32}x_{31}y_{13}x_{32}y_{23})^{p-1}\] with coefficient 1. This is true since there is a unique way to obtain this term. Thus by Fedder's criterion \cite{fedder} (Theorem 1.12), $R/\overline{\mathcal{I}}$ is $F$-pure and so is $R/\mathcal{I}$. 
\end{proof}

\begin{thm} Let $n=2k+1$ be an odd integer with $k \geq 1$ and let $K$ be a field of positive prime characteristic $p$. Then $R/\mathcal{I}$ is an $F$-pure ring.  \end{thm}

\begin{proof}

We prove by induction on $k$. Base of the induction is proved in the previous proposition. 

Now let us assume that for all odd integers $n\leq 2k-1$,  after we factor out by part of a system of parameters from Theorem \ref{SOPodd} we have that $R/\overline{\mathcal{I}}$ is $F$-pure, that is, a deformation of $R/\mathcal{I}$ is $F$-pure.  

Factor out $R$ by part of a system of parameters from Theorem \ref{SOPodd}. Then we obtain the following matrices 
\begin{small}\[\overline{X}=\left[\begin{array}{c|ccc|c}
  0& 0&\ldots &0&0 \\ \hline
  0 &&&&0\\
  \vdots &&&& \vdots \\
     0 &&&& 0 \\
x_{k+1,1} &&& & x_{k+1, 2k+1} \\
  x_{k+2,1} &&\raisebox{-5pt}{{\Large \mbox{{$X_0$}}}}&& 0 \\
   0 &&&& 0 \\
  \vdots &&&& \vdots \\
   0 &&&& 0 \\
  \hline
   0&0 \ldots 0&x_{2k+1, k+1}&0 \ldots 0&0\\ \end{array} \right] \]
and 
\[\overline{Y}=\left[\begin{array}{c|ccc|c}
 0&0 \ldots 0&y_{1, k+1} \quad y_{1, k+2}&0 \ldots 0&0\\  
 \hline
    0 &&&&0\\
\vdots &&&& \vdots \\
     0 &&&& 0 \\
0 & &\raisebox{-5pt}{{\Large \mbox{{$Y_0$}}}}&& y_{k+1, 2k+1} \\
   0 &&&& 0 \\
  \vdots &&&& \vdots \\
   0 &&&& 0 \\
  \hline
   0&0 \ldots 0&y_{2k+1, k+1} \quad 0&0 \ldots 0&0\\ \end{array} \right] \]
\end{small}

Therefore, the commutator matrix is
\begin{small}\[\overline{XY-YX}=\]\[\left[\begin{array}{c|c|c}
 -x_{k+1,1}y_{1,k+1}-x_{k+2, 1}y_{1,k+2}&\cdots &-x_{k+1, 2k+1}y_{1, k+1}\\
\hline
    &&\\
\vdots  & \raisebox{-5pt}{{\large \mbox{{$\overline{C}_0+Z$}}}}& \vdots  \\
    &&  \\
  \hline
   -x_{k+1, 1}y_{2k+1, k+1}&\cdots&x_{2k+1, k+1}y_{k+1, 2k+1}-x_{k+1, 2k+1}y_{2k+1, k+1}\\ \end{array} \right] \]
\end{small}
where the entries of the matrix $Z$ are in the ideal  \[\mathcal{L}=(x_{k+1, 2k+1}, x_{k+1, 1}, x_{k+2, 1},  x_{2k+1, k+1}, y_{1, k+1},  y_{2k+1, k+1},  y_{1,k+2}, y_{k+1, 2k+1}).\] 
Now the generators of $\overline{\mathcal{I}}$ are the entries of the main diagonal and anti-diagonal of $\overline{C}_0+Z$ together with the entries at the corners of the above commutator matrix.  Denote by $\Pi$ the product of the $(2k-2)$ entries of the main diagonal and $2k-2$ entries of the anti-diagonal of $\overline{C}_0+Z$. 
Then  \[\overline{\mathcal{I}}^{[p]}:\overline{\mathcal{I}}= \left(x_{k+1, 2k+1}y_{1, k+1}x_{k+1, 1}y_{2k+1, k+1}(x_{k+1,1}y_{1,k+1}-x_{k+2, 1}y_{1,k+2})\right)^{p-1}\cdot\]\[\left((x_{2k+1, k+1}y_{k+1, 2k+1}-x_{k+1, 2k+1}y_{2k+1, k+1})\Pi\right)^{p-1}+\overline{\mathcal{I}}^{[p]}.\]  

By induction hypothesis, if we factor out  by the ideal $\mathcal{L}$, $(\Pi)^{p-1}$ has a monomial term in the entries of $\overline{X}_0$, $\overline{Y}_0$ with non-zero coefficient (equal to 1) modulo $p$ so that every indeterminate that appears in it has degree $p-1$. Therefore, $(\Pi)^{p-1}$ also has such a term in the entries of $\overline{X}_0$ and $\overline{Y}_0$. Denote it by $\mu$. 

We claim that \[\mu (x_{k+1, 1}x_{k+1, 2k+1}x_{k+2, 1} x_{2k+1, k+1}y_{1, k+1} y_{1,k+2}  y_{k+1, 2k+1}y_{2k+1, k+1})^{p-1}\] is a nonzero monomial term of a generator $\overline{\mathcal{I}} ^{[p]}:\overline{\mathcal{I}}$. Clearly, such a term exists. It has a non-zero coefficient since it can be obtained in a unique way from $(x_{k+1, 2k+1}y_{1, k+1}x_{k+1, 1}y_{2k+1, k+1}(x_{k+1,1}y_{1,k+1}-x_{k+2, 1}y_{1,k+2})(x_{2k+1, k+1}y_{k+1, 2k+1}-x_{k+1, 2k+1}y_{2k+1, k+1})\Pi)^{p-1}$.  Therefore, $R/\overline{\mathcal{I}}$ is $F$-pure by Fedder's criterion and hence so is $R/\mathcal{I}$. 
\end{proof}

Next we focus on the case of matrices of even size. 

\begin{lem}\label{binom} Let $p>2$ be any prime number. Then \[\sum_{b=0}^{\frac{p-1}{2}}(-1)^b\sum_{a=b}^{p-1-b}{{a+b}\choose {a}}{{a}\choose{b}}\equiv _{mod \, p}(-1)^{\frac{p-1}{2}}.\]
\end{lem}
\begin{proof}
Let $A_b=\sum_{a=b}^{p-1-b}{{a+b}\choose {a}}{{a}\choose{b}}$. We prove the lemma by showing that $A_b\equiv_{mod \, p}0$ for all $0\leq b\leq (p-3)/2$ and $A_{(p-1)/2}\equiv_{mod \, p}(-1)^{(p-1)/2}$.
\[A_ b= \sum_{a=b}^{p-1-b}\frac{(a+b)!}{b!b!(a-b)!}=\]
\[{{2b}\choose{b}}\left(\sum_{a=b+1}^{p-1-b}\frac{(2b+1)(2b+2)\ldots (a+b)}{(a-b)!}+1\right)=\]
\[{{2b}\choose{b}}\sum_{c=0}^{p-1-2b}{{2b+c}\choose {2b}}=\]
\[{{2b}\choose{b}}\sum_{d=0}^{p-1-2b}{{p-1-d}\choose {2b}}.\]

\begin{claim} For any prime number $p$ we have that 
\[{{p-i}\choose {k}}\equiv_{mod \, p}(-1)^{i-1}{{p-1-k}\choose{i-1}}{{p-1}\choose{k}}.\] for all $1\leq i\leq p-1$ and all $0\leq k\leq p-i$. 
\end{claim}

We prove the claim by induction on $i$. 
If $i=1$, the claim is clear. 
We have that  \[{{p-i-1}\choose {k}}={{p-i}\choose {k}}\frac{p-i-k}{p-i}\]
by induction hypothesis modulo $p$ this equals to
\[(-1)^{i-1}{{p-1-k}\choose {i-1}}{{p-1}\choose {k}}\frac{p-i-k}{-i}=(-1)^i{{p-1-k}\choose {i}}{{p-1}\choose {k}}\]

Now we are ready to finish the proof of the lemma. Let $0\leq b<(p-3)/2$
\[A_b={{2b}\choose{b}}\sum_{d=0}^{p-1-2b}{{p-1-d}\choose {2b}}\equiv_{mod \, p}\]
\[{{2b}\choose{b}}{{p-1}\choose{2b}}\sum_{d=0}^{p-1-2b}(-1)^{d+1}{{p-1-2b}\choose {d}}=\]
Using Newton's binomial theorem we obtain
\[-{{2b}\choose{b}}{{p-1}\choose{2b}}(1-1)^{p-1-2b}=0.\]
And 
\[A_{(p-1)/2}={{p-1}\choose{(p-1)/2}}\equiv_{mod \, p}(-1)^{(p-1)/2}. \]

\end{proof}

\begin{prop} Let $n=4$ and let $K$ be a field of positive prime characteristic $p\neq 2$. Then $R/\mathcal{I}$ is $F$-pure. \end{prop}
\begin{proof} 
 
After factoring $R$ by part of a system of parameters from Theorem \ref{SOPeven}  we have the following matrices 
\[\overline{X}=\left[\begin{array}{cccc}
0& 0&0&0\\
x_{21}& 0&0&x_{24}\\
x_{31}&x_{32}&0&x_{34}\\
0&x_{42}&x_{43}&0\end{array}\right] \text{ and } \overline{Y}=\left[\begin{array}{cccc}
0& y_{12}&y_{13}&0\\
0& 0&y_{23}&y_{24}\\
0&0&0&y_{34}\\
0&y_{42}&y_{43}&0\\
\end{array}\right]\]
Therefore, by computing the commutator matrix, we have that the product of the generators of $\overline{\mathcal{I}}$ is 
\[\omega=(x_{21}y_{12}+x_{24}y_{42}-x_{32}y_{42}-x_{42}y_{24})\]
\[(x_{21}y_{13}+x_{24}y_{43}-x_{43}y_{24})(x_{31}y_{12}+x_{34}y_{42}-x_{42}y_{34})\]
\[(x_{31}y_{13}+x_{32}y_{23}+x_{34}y_{43}-x_{43}y_{34})(x_{21}y_{12}+x_{31}y_{13})\]
\[(x_{24}y_{12}+x_{34}y_{13})(x_{21}y_{42}+x_{31}y_{43}).\]
and since $\mathcal{I}$ is a complete intersection,
\[\overline{\mathcal{I}}^{[p]}:\overline{\mathcal{I}}=(\omega)^{p-1}+\overline{\mathcal{I}}^{[p]}.\]

To prove $F$-purity of $R/\mathcal{I}$ is sufficient to show that $\omega^{p-1}$ has a non-zero monomial term which does not lie in $\mathfrak{m}^{[p]} $, that is, its indeterminates have degree at most $p-1$.

We claim that $\omega^{p-1}$ has a monomial term which is the product of all the indeterminates  in $\overline{X}$ and $\overline{Y}$ raised to the power $p-1$ with non-zero coefficient modulo $p$. Using Newton's multinomial theorem,  we obtain that each term of $\omega^{p-1}$ has the form
\[{{p-1}\choose {\alpha_1 \quad \beta_1 \quad \gamma_1\quad\delta_1}}(x_{21}y_{12})^{\alpha_1}(x_{24}y_{42})^{\beta_1}(-x_{32}y_{42})^{\gamma_1}(-x_{42}y_{24})^{\delta_1}\] 
\[{{p-1}\choose {\alpha_2 \quad \beta_2 \quad \gamma_2}} (x_{21}y_{13})^{\alpha_2}(x_{24}y_{43})^{\beta_2}(-x_{43}y_{24})^{\gamma_3} \] 
\[{{p-1}\choose {\alpha_3 \quad \beta_3 \quad \gamma_3}}(x_{31}y_{12})^{\alpha_3}(x_{34}y_{42})^{\beta_3}(-x_{42}y_{34})^{\gamma_3}\] 
\[{{p-1}\choose {\alpha_4 \quad \beta_4 \quad \gamma_4\quad\delta_4}}(x_{31}y_{13})^{\alpha_4}(x_{32}y_{23})^{\beta_4}(x_{34}y_{43})^{\gamma_4}(-x_{43}y_{34})^{\delta_4}\] 
\[{{p-1}\choose {\alpha_5 \quad \beta_5}}(x_{21}y_{12})^{\alpha_5}(x_{31}y_{13})^{\beta_5}\] 
\[{{p-1}\choose {\alpha_6 \quad \beta_6}}(x_{24}y_{12})^{\alpha_6}(x_{34}y_{13})^{\beta_6}\] 
\[{{p-1}\choose {\alpha_7 \quad \beta_7}}(x_{21}y_{42})^{\alpha_7}(x_{31}y_{43})^{\beta_7}\] 

Denote by $A_{ij}$ and $B_{ij}$ the degrees of $x_{ij}$ and $y_{ij}$, respectively. Then 
\begin{multicols}{2}

$A_{21}=\alpha_1+\alpha_2+\alpha_5+\alpha_7$

$A_{24}=\alpha_6+\beta_1+\beta_6$

$A_{31}=\alpha_3+\alpha_4+\beta_5+\beta_7$

$A_{32}=\beta_4+\gamma_1$

$A_{34}=\beta_3+\beta_6+\gamma_4$

$A_{42}=\gamma_3+\delta_1$

$A_{43}=\gamma_2+\delta_4$

$B_{12}=\alpha_1+\alpha_3+\alpha_5+\alpha_6$

$B_{13}=\alpha_2+\alpha_4+\beta_5+\beta_6$

$B_{23}=\beta_4+\gamma_1$

$B_{24}=\gamma_2+\delta_1$

$B_{34}=\gamma_3+\delta_4$

$B_{42}=\alpha_7+\beta_1+\beta_3$

$B_{43}=\beta_2+\beta_7+\gamma_4$

\end{multicols}

In addition, let 

$$C_1=\alpha_1+\beta_1+\gamma_1+\delta_1$$
$$C_2=\alpha_2+\beta_2+\gamma_2$$
$$C_3=\alpha_3+\beta_3+\gamma_3$$
$$C_4=\alpha_4+\beta_4+\gamma_4+\delta_4$$
$$C_5=\alpha_5+\beta_5$$
$$C_6=\alpha_6+\beta_6$$
$$C_7=\alpha_7+\beta_7$$

Therefore, we consider a system of linear equations $A_{ij}=p-1$, $B_{ij}=p-1$, $C_k=p-1$ for all the  values of $ i, j, k$.
We look for all solutions in the set of non-negative integers. 
First, observe that the system cannot have a unique solution as the sum of the equations defined by $A_{ij}$ is equal to the sum of the equations defined by $B_{ij}$.  Next, 
\[A_{21}+A_{31}=\alpha_1+\alpha_2+\alpha_5+\alpha_7+\alpha_3+\alpha_4+\beta_5+\beta_7=\alpha_1+\alpha_2+\alpha_3+\alpha_4+C_5+C_7\]
Therefore, $\alpha_1+\alpha_2+\alpha_3+\alpha_4=0$. Since  $\alpha_1,\alpha_2,\alpha_3,\alpha_4$ are non-negative integers, we must have that each term is 0. That is, $\alpha_i=0$ for all $1\leq i \leq 4$. Then using standard techniques from linear algebra, we obtain the following solution to our system of linear equations. 

\[(\alpha_1, \, \beta_1, \, \gamma_1,\,\delta_1)=(0,\, \beta_1, \, p-1-(\beta_1+\beta_2), \, \beta_2)\]
\[(\alpha_2, \, \beta_2, \, \gamma_2)=(0,\,\beta_2,\, p-1-\beta_2)\]
\[(\alpha_3, \, \beta_3, \, \gamma_3)=(0,\,\beta_2, \,p-1-\beta_2)\]
\[(\alpha_4, \, \beta_4, \, \gamma_4,\,\delta_4)=(0,\,\beta_1, \,p-1-(\beta_1+2\beta_2), \,\beta_2)\]
\[(\alpha_5, \, \beta_5)=(\beta_1+\beta_2,\, p-1-(\beta_1+\beta_2))\]
\[(\alpha_6, \, \beta_6)=(p-1-(\beta_1+\beta_2), \,\beta_1+\beta_2)\]
\[(\alpha_7, \, \beta_7)=(p-1-(\beta_1+\beta_2), \, \beta_1+\beta_2)\]

Therefore, the coefficient of our monomial term is 

\[\sum_{\beta_2=0}^{\frac{p-1}{2}}\sum_{\beta_1=0}^{p-1-2\beta_2}(-1)^{\beta_1+\beta_2}{{p-1}\choose {\beta_1\quad\beta_2 \quad p-1-( \beta_1+\beta_2)}}{{p-1}\choose{\beta_2}}^2{{p-1}\choose{\beta_1+\beta_2}}^3\cdot \]\[{{p-1}\choose{\beta_1+\beta_2\quad \beta_2\quad p-1-\beta_1-2\beta_2}}\equiv_{mod \, p}\]
\[\sum_{\beta_2=0}^{\frac{p-1}{2}}\sum_{\beta_1=0}^{p-1-2\beta_2}(-1)^{\beta_1+\beta_2}{{p-1}\choose{\beta_1+\beta_2}}{{p-1}\choose {\beta_1\quad\beta_2 \quad p-1-( \beta_1+\beta_2)}}\cdot\]\[{{p-1}\choose{\beta_1+\beta_2\quad \beta_2\quad p-1-\beta_1-2\beta_2}}\equiv_{mod \, p}\]
\[\sum_{\beta_2=0}^{\frac{p-1}{2}}\sum_{\beta_1=0}^{p-1-2\beta_2}(-1)^{\beta_1+\beta_2}{{p-1}\choose{\beta_1+\beta_2}}^2{{p-1}\choose{\beta_1\quad\beta_2\quad \beta_2\quad p-1-\beta_1-2\beta_2}}\equiv_{mod \, p}\]
\[\sum_{\beta_2=0}^{\frac{p-1}{2}}\sum_{\beta_1=0}^{p-1-2\beta_2}(-1)^{\beta_1+\beta_2}{{p-1}\choose{\beta_1\quad\beta_2\quad \beta_2\quad p-1-\beta_1-2\beta_2}}\equiv_{mod \, p}\]
\[\sum_{\beta_2=0}^{\frac{p-1}{2}}\sum_{\beta_1=0}^{p-1-2\beta_2}(-1)^{\beta_1+\beta_2}{{p-1}\choose {\beta_1+2\beta_2}}{{\beta_1+2\beta_2}\choose{\beta_1\quad\beta_2\quad \beta_2}}\equiv_{mod \, p}\]
\[\sum_{\beta_2=0}^{\frac{p-1}{2}}\sum_{\beta_1=0}^{p-1-2\beta_2}(-1)^{\beta_2}{{\beta_1+2\beta_2}\choose{\beta_1\quad\beta_2\quad \beta_2}}\equiv_{mod \, p}\]

Make a substitution $\beta_1+\beta_2=a$, $\beta_2=b$,  then we have 
\[\sum_{b=0}^{\frac{p-1}{2}}(-1)^b\sum_{a=b}^{p-1-b}{{a+b}\choose {a}}{{a}\choose{b}}\]
which is equal to $(-1)^{\frac{p-1}{2}}$ modulo $p$ due to Lemma \ref{binom}. 

Thus, we have that the coefficient of our monomial term is equal to $\pm 1$ modulo $p$ and $R/\mathcal{I}$ is $F$-pure by Fedder's criterion. \\
\end{proof}

\begin{prop} Let $n=4$ and let $K$ be a field of characteristic 2. Then $R/\mathcal{I}$ is $F$-pure. 
\end{prop}
\begin{proof}
First, we factor $R$ by part of a system of parameters from Theorem \ref{SOPchar2}. Then \[\overline{\mathcal{I}}=\left(x_{31}y_{13}+x_{41}y_{14}, x_{32}y_{23}+x_{42}y_{24}+x_{24}y_{42}, x_{31}y_{13}+x_{32}y_{23}+x_{43}y_{34}+x_{34}y_{43},\right.\]\[ \left.x_{34}y_{13}, x_{43}y_{24}+x_{24}y_{43}, x_{42}y_{34}+x_{34}y_{42}, x_{31}y_{43}\right).\] Since $\mathcal{I}$ is a complete intersection, we have $\overline{\mathcal{I}}^{[2]}:\overline{\mathcal{I}}=(\omega)+\overline{\mathcal{I}}^{[2]}$, where $\omega$ is the product of the generators of $\overline{\mathcal{I}}$. Using Macaulay2 \cite{macaulay2}, we compute this product and find that there is exactly one monomial term such that every indeterminates in its support has degree 1: \[x_{24}x_{31}x_{32}x_{34}x_{41}x_{42}x_{43}y_{13}y_{14}y_{23}y_{24}y_{34}y_{42}y_{43}.\]
Thus, by Fedder's criterion, $R/\overline{\mathcal{I}}$ is $F$-pure.\\
\end{proof}

\begin{thm} Let $n=2k$ be an even integer with $k\geq 2$ and let $K$ be a field of positive prime characteristic $p$. Then $R/\mathcal{I}$ is $F$-pure.  \end{thm}
\begin{proof}
We prove by induction on $k$. Base of the induction was proved in the previous two propositions. 
We assume that for all $n\leq 2k-2$ after we factor by part of a system of parameters from Theorem \ref{SOPeven} and Theorem \ref{SOPchar2} we have that $R/\overline{\mathcal{I}}$ is $F$-pure and if $\omega$ is the product of the generators of $\overline{\mathcal{I}}$, then $\omega^{p-1}$ has a monomial term $\mu \notin \mathfrak{m}^{[p]}$. 
 Then we have that

\begin{small}\[
\overline{X}=
\left[\begin{array}{c|ccc|c}
  0& 0&\ldots& 0&0 \\ \hline
  0 &&&&0\\
  \vdots &&&& \vdots \\
     0 &&&& 0 \\
x_{k1} & &&& x_{k, 2k} \\
  x_{k+1,1} &&\raisebox{-5pt}{{\Large \mbox{{$X_0$}}}}&& 0 \\
   0 &&&& 0 \\
  \vdots &&&& \vdots \\
   0 &&&& 0 \\
  \hline
   0&\, \, 0 \ldots &0 \, x_{2k, k}\, 0& \ldots \, \, 0&0\\ \end{array} \right]\]
and 
\[\overline{Y}=\left[\begin{array}{c|ccc|c}
 0&0 \ldots 0&y_{1k}\quad y_{1, k+1}&0 \ldots 0&0\\  
 \hline
    0 &&&&0\\
\vdots &&&& \vdots \\
     0 &&&& 0 \\
0 & &\raisebox{-5pt}{{\Large \mbox{{$Y_0$}}}}&& y_{k, 2k} \\
   0 &&&& 0 \\
   0 &&&& 0 \\
  \vdots &&&& \vdots \\
   0 &&&& 0 \\
  \hline
   0&0 \ldots 0\,&y_{2k, k}\quad  0&0 \ldots 0&0\\ \end{array} \right] \]
\end{small}

Therefore, the commutator matrix is
\begin{small}\[\overline{XY-YX}=\]\[\left[\begin{array}{c|c|c}
 -x_{k1}y_{1k}-x_{k+1, 1}y_{1, k+1}&\cdots &-x_{k, 2k}y_{1k}\\
\hline
    &&\\
\vdots  & \raisebox{-5pt}{{\large \mbox{{$\overline{C}_0+Z$}}}}& \vdots  \\
    &&  \\
  \hline
   -x_{k1}y_{2k, k}&\cdots&x_{2k, k}y_{k, 2k}-x_{k, 2k}y_{2k, k} \\ \end{array} \right] \]
\end{small}
where the entries of the matrix $Z$ are in the ideal  \begin{align*}
 &\mathcal{L}=(x_{k1}, x_{k+1, 1}, x_{k, 2k}, x_{2k, k}, y_{1k},y_{1, k+1}, y_{k, 2k}, y_{2k, k}).   
\end{align*} 

By induction hypothesis, the $(p-1)$st power of the product of any $2k-3$ entries on the main diagonal and all the entries of the main antidiagonal of $\overline{C}_0+Z$ has a non-zero monomial term $\mu$ in the entries of $\overline{X}_0, \overline{Y}_0$ which is not in $\mathfrak{m}^{[p]}$, (it exists after we factor out by the ideal $\mathcal{L}$). Therefore, the $(p-1)$st power of the generators of $\overline{\mathcal{I}}$ has a non-zero monomial term 
$$\mu(x_{k1}x_{k+1, 1}x_{k, 2k}x_{2k, k} y_{1k}y_{1, k+1} y_{k, 2k}y_{2k, k})^{p-1}\notin \mathfrak{m}^{[p]}.$$ 
\end{proof}
\section{Zero diagonal commutator}

In this section we study the algebraic set defined by the vanishing of the diagonal of the commutator matrix $C$.
Recall that $I$ is the ideal of $R$ generated by the entries of the main diagonal of $C$. Since the trace of $C$ is 0, we can choose the generators of $I$ to be the first $n-1$ entries of the main diagonal counting from the upper left corner. 

In his thesis \cite{young},H-W.Young has proved that $I$ is a complete intersection. We recover his result from our proof that $\mathcal{I}$ is generated by a regular sequence. 
\begin{thm}
\cite{young} (Theorem 5.3.1) $R/I$ is a complete intersection ring of dimension $2n^2-n+1$.
\end{thm}

Moreover, Young showed the following properties of $I$.
\begin{thm}
\begin{itemize}
    \item[1.] \cite{young} (Theorem 5.3.2) $R/I$ is reduced when the characteristic of the field $K$ is 0.
    \item[2.] \cite{young} (Theorem 5.3.3) $R/I$ is irreducible when $n=2$ and $n=3$ in all characteristics.Therefore, $R/I$ is a domain in characteristic 0 when $n=2$ and $n=3$.
\end{itemize}
\end{thm}

Next we find a homogeneous system of parameters for $R/I$ and show that $R/I$ is $F$-regular when $K$ is a field of positive prime characteristic $p$. 

\subsection{A system of parameters}

We already have one system of parameters on $R/I$ which is obtained from the one we got for $R/\mathcal{I}$. However, the one we offer here is simpler and will allow us to prove $F$-regularity of $R/I$.

\begin{thm}\label{SOP-I} A set of elements 
\[ x_{11}\]
\[x_{ij}, \text{ where } 2\leq i\leq n, \, 1\leq j\leq n,\]
\[y_{ij}, \text{ where } 1\leq i\leq n, \, 2\leq j\leq n,\]
\[x_{1j}-y_{j1}, \text{ where } 1\leq j\leq n,\]
is a homogeneous system of parameters and a regular sequence on $R/I$.
\end{thm}
\begin{proof}
The dimension of $R/I$ is $2n^2-n+1$ and the number of elements in the above set is also $2n^2-n+1$. It is necessary and sufficient to prove that if we factor out by the ideal they generate together with $I$, the resulting ring has Krull dimension 0. We obtain that this factor ring is isomorphic to the following ring
\[\frac{K[x_{12}, x_{13}, \ldots, x_{1n}]}{(x_{12}^2, x_{13}^2, \ldots, x_{1n}^2)},\]
which is clearly 0-dimensional.
\end{proof}

\subsection{F-regularity}

\begin{thm}
Let $K$ be a field of positive prime characteristic $p$. Then $R/I$ is $F$-regular. 
\end{thm}
\begin{proof}
In this proof we use the fact that F-regularity deforms for Gorenstein rings \cite{smooth} (Corollary 4.7).
We factor $R/I$ by part of a system of parameters from above: we annihilate $x_{11}$, $y_{11}$, the entries of $X$ below the first row and the entries of $Y$ to the right of the first column. Denote the factor ring $\overline{R}$.

Let 
\[Z_i=
\begin{bmatrix}
x_{i,1} & x_{1,i}\\
y_{i,1} & y_{1,i}
\end{bmatrix}
\] where $2\leq i\leq n$.

Then we have that 
\[\overline{R}\simeq \frac{K[\{Z_i\}_{i=2}^n]}{(\{\det Z_i\}_{i=2}^n)} \simeq \bigotimes\limits_{i=2}^n \frac{K[Z_i]}{(\det Z_i)},
\]
where $K[Z_i]/(\det Z_i)$ is a determinantal ring, known to be F-regular, \cite{det}, Theorem 7.14. Using  Theorem 7.45 in \cite{smooth}, we obtain that $\overline{R}$ is F-regular, and hence so is $R/I$. 
\end{proof}
\begin{cor}
$R/I$ is a domain and hence $I$ is a prime ideal when $K$ is a field of positive prime characteristic. 
\end{cor}
Next we observe the following result. 

\begin{lem} The ring $R/I$ deforms to a Staynley-Reisner ring. \end{lem}
\begin{proof}
Factor $R/I$ by part of a system of parameters from Theorem \ref{SOP-I}. Then we have that the image of $I$ is generated by the square-free monomials $\{x_{1j}y_{j1}\,| \,  2\leq j \leq n\}$. 
\end{proof}

\section{Zero anti-diagonal commutator}
In this section we study the algebraic set defined by the vanishing of the anti-diagonal entries of the commutator matrix $C$. Recall that the ideal generated by these entries is denoted by $J$.
\begin{thm}
$R/J$ is a complete intersection ring of dimension $2n^2-n$.
\end{thm}
\begin{proof}
Since $\mathcal{I}$ is generated by a regular sequence, then so is $J$.
\end{proof}
Next we find a system of parameters for $R/J$. We already have one obtained from a system of parameters for $R/\mathcal{I}$. However, the one we give here is simpler.
\begin{thm}\label{SOP-J}
A set of elements 
\[x_{ij}, \text{ where } 2\leq i\leq n, \, 1\leq j\leq n,\]
\[y_{ij}, \text{ where } 1\leq i\leq n, \, 2\leq j\leq n,\]
\[x_{1j}-y_{n-j+1,j}, \text{ where } 1\leq j\leq n\]
is a homogeneous system of parameters and a regular sequence on $R/J$.
\end{thm}
\begin{proof}
If we factor out by the ideal generated by the set elements, we obtain the following matrices
\[\overline{X}=\left[\begin{array}{cccc}
x_{11}&x_{12}&\ldots&x_{1n}\\
0& 0&\ldots&0 \\
\ldots\\
0& 0&\ldots&0 \\
\end{array}\right] \text{ and } \overline{Y}=\left[\begin{array}{cccc}
x_{1n}&0&\ldots&0 \\
x_{1,n-1}&0&\ldots&0 \\
\ldots\\
x_{11}&0&\ldots&0 \\
\end{array}\right]\]
Then the generators for the image of $J$ are $x_{11}^2, x_{12}^2, \ldots, x_{1n}^2$. Thus, the factor ring of $R/J$ by the ideal generated by the set elements has Krull dimension 0. Hence the result. 
\end{proof}
\begin{thm}
Let $K$ be a field of positive prime characteristic $p$. Then $R/J$ is $F$-pure. 
\end{thm}
\begin{proof}
We have that $\mathcal{I}=I+J$ is generated by a regular sequence, therefore, $R/\mathcal{I}$ is a deformation of $R/J$. Since $R/\mathcal{I}$ is $F$-pure, then so is $R/J$.
\end{proof}
\begin{cor}
$R/J$ is a reduced ring, that is, $J$ is a radical ideal, when $K$ is a field of positive prime characteristic.
\end{cor}
Here we also observe that the ring $R/J$ has a deformation to a Staynley-Reisner ring. 
\begin{lem} The ring $R/J$ deforms to a Staynley-Reisner ring. \end{lem}
\begin{proof}
Factor $R/J$ by part of a system of parameters from Theorem \ref{SOP-J}. Then we have that the image of $J$ is generated by the square-free monomials $\{x_{1j}y_{j1}\, |\, 1\leq j \leq n\}$. 
\end{proof}
\section{Conjectures}
In this section we state conjectures that naturally arise in the context of the algebraic sets defined by the ideals $\mathcal{I}$, $I$ and $J$.
\begin{conj}
Let $K$ be a field of positive prime characteristic. Then $R/\mathcal{I}$ and $R/J$ are $F$-regular rings. In particular, they are integral domains.
\end{conj} 

\begin{conj}
The ideals $\mathcal{I}$, $I$ and $J$ are prime ideals in all characteristics.
\end{conj}
\section{Appendix}

The aim of this section is to display how matrices $X$ and $Y$ look like after the ring $R$ is factored out by a system of parameters from Theorems \ref{SOPodd}, \ref{SOPeven}, \ref{SOPchar2}. The base cases $n=3$ and $n=4$ are respectively highlighted in the center of each of the matrices.  \\

When $n=7$, we have 
\begin{small}\[\overline{X}=\left[
\begin{array}{cc|ccc|cc}
0&0&0&0&0&0&0\\
0&0&0&0&0&0&0\\
\hline
0&0&0&0&0&0&0\\
x_{41}&x_{42}&x_{43}&0&x_{45}&x_{46}&x_{47}\\
x_{51}&x_{52}&x_{53}&x_{54}&0&0&0\\
\hline
0&0&0&x_{64}&0&0&0\\
0&0&0&x_{74}&0&0&0\\ 
\end{array} \right]\]
and 
\[\overline{Y}=\left[
\begin{array}{cc|ccc|cc}
0&0&0&x_{47}&x_{51}&0&0\\
0&0&0&x_{46}&x_{52}&0&0\\
\hline
0&0&0&x_{45}&x_{53}&0&0\\
0&0&0&0&x_{54}&x_{64}&x_{74}\\
0&0&0&x_{43}&0&0&0\\
\hline
0&0&0&x_{42}&0&0&0\\
0&0&0&x_{41}&0&0&0\\ 
\end{array} \right]\]
\end{small} 

When $n=8$ and the characteristic of $K$ is not 2, we have

\begin{small}\[\overline{X}=\left[
\begin{array}{cc|cccc|cc}
0&0&0&0&0&0&0&0\\
0&0&0&0&0&0&0&0\\
\hline
0&0&0&0&0&0&0&0\\
x_{41}&x_{42}&x_{43}&0&0&x_{46}&x_{47}&x_{48}\\
x_{51}&x_{52}&x_{53}&x_{54}&0&x_{56}&0&0\\
0&0&0&x_{64}&x_{65}&0&0&0\\
\hline
0&0&0&x_{74}&0&0&0&0\\ 
0&0&0&x_{84}&0&0&0&0\\ 
\end{array} \right]\]
and 
\[\overline{Y}=\left[
\begin{array}{cc|cccc|cc}
0&0&0&x_{48}&x_{51}&0&0&0\\
0&0&0&x_{47}&x_{52}&0&0&0\\
\hline
0&0&0&x_{46}&x_{53}&0&0&0\\
0&0&0&0&x_{54}&x_{65}&x_{74}&x_{84}\\
0&0&0&0&0&x_{64}&0&0\\
0&0&0&x_{43}&x_{56}&0&0&0\\
\hline
0&0&0&x_{42}&0&0&0&0\\
0&0&0&x_{41}&0&0&0&0\\ 
\end{array} \right]\]
\end{small} 

When $n=8$ and the characteristic of $K$ is 2, we have

\begin{small}\[\overline{X}=\left[
\begin{array}{cc|cccc|cc}
0&0&0&0&0&0&0&0\\
0&0&0&0&0&0&0&0\\
\hline
0&0&0&0&0&0&0&0\\
x_{41}&x_{42}&0&0&0&x_{46}&x_{47}&x_{48}\\
x_{51}&x_{52}&x_{53}&x_{54}&0&x_{56}&0&0\\
0&0&x_{63}&x_{64}&x_{65}&0&0&0\\
\hline
0&0&0&x_{74}&0&0&0&0\\ 
0&0&0&x_{84}&0&0&0&0\\ 
\end{array} \right]\]
and 
\[\overline{Y}=\left[
\begin{array}{cc|cccc|cc}
0&0&0&x_{48}&x_{51}&0&0&0\\
0&0&0&x_{47}&x_{52}&0&0&0\\
\hline
0&0& 0&0&x_{56}&x_{63}&0&0\\
0&0& 0&0&x_{54}&x_{65}&x_{74}&x_{84}\\
0&0& 0&0&0&x_{64}&0&0\\
0&0& 0&x_{46}&x_{53}&0&0&0\\
\hline
0&0&0&x_{42}&0&0&0&0\\
0&0&0&x_{41}&0&0&0&0\\ 
\end{array} \right]\]
\end{small} 

\bibliographystyle{plain}
\bibliography{bibliography}

\end{document}